\documentclass{birkjour}

\usepackage{enumitem}
\usepackage{tikz}
\usepackage{comment}
\usepackage{amsmath}
\usepackage{stmaryrd}
\usepackage{bbm}
\usepackage{esint}
\usepackage{marginnote}
\usepackage{amssymb}
\usepackage[hidelinks]{hyperref}
\usepackage{url}

\renewcommand{\L}[1]{\mathbf{L^{\pmb #1}}}
\newcommand{\C}[1]{\mathbf{C^{\pmb #1}}}
\newcommand{\Cc}[1]{\mathbf{C_{\rm c}^{#1}}}
\newcommand{\Lloc}[1]{\mathbf{L_{\textbf{loc}}^{\pmb #1}}}
\renewcommand{\geq}{\geqslant}
\renewcommand{\leq}{\leqslant}
\DeclareMathOperator{\TV}{TV}
\DeclareMathOperator{\sign}{sign}

\newcommand{\disint}[1]{\llbracket #1 \rrbracket}
\newcommand{\Z}{\mathbb Z}

\newtheorem{thm}{Theorem}[section]
\newtheorem{cor}[thm]{Corollary}
\newtheorem{lem}[thm]{Lemma}
\newtheorem{prop}[thm]{Proposition}
\theoremstyle{definition}
\newtheorem{defn}[thm]{Definition}
\theoremstyle{remark}

\numberwithin{equation}{section}

\begin{document}
\allowdisplaybreaks


\title[On stability of Hughes' dynamics]{On stability of one-dimensional\\Hughes' dynamics with affine costs}

\author[B.\ Andreianov]{Boris Andreianov}
\address{Institut Denis Poisson CNRS UMR7013, Universit\'e de Tours, Universit\'e d'Orl\'eans,\\ Parc Grandmont, 37200 Tours, France\\ and\\
	Peoples’ Friendship University of Russia (RUDN University)\\
	6 Miklukho-Maklaya St, Moscow, 117198, Russian Federation}
\email{boris.andreianov@univ-tours.fr}

\author[S.\ Fagioli]{Simone Fagioli}

\address{Department of Information Engineering, Computer Science and Mathematics, University of L'Aquila I-67100, Italy}
\email{simone.fagioli@univaq.it}

\author[M.D.~Rosini]{Massimiliano~D.~Rosini}

\address{Uniwersyty of Marii Curie-Sk\l odowskiej, Plac Marii Curie-Sk\l odowskiej 1 20-031 Lublin, Poland\\
Department of Management and Business Administration,
University 
of Chieti-Pescara, viale Pindaro, 42, Pescara, 65127, Italy}
\email{massimiliano.rosini@mail.umcs.pl}

\author[G.\ Stivaletta]{Graziano Stivaletta}

\address{Department of Information Engineering, Computer Science and Mathematics, University of L'Aquila I-67100, Italy}
\email{graziano.stivaletta@graduate.univaq.it}


\begin{abstract}
We investigate stability issues for the one-dimensional variant of the celebrated Hughes model for pedestrian evacuation. The cost function is assumed to be affine, which is a setting where existence of solutions with $ \mathbf{BV}_{\textbf{loc}}$ in space regularity, away from the so-called turning curve, was recently established. We provide a uniqueness result for solutions having the special property that agents never cross the turning curve (which implies that they are $ \mathbf{BV}$ globally). In the same setting, continuous dependence of solutions on the cost parameter is highlighted. {On the other hand, numerical simulations using the many-particle approximation of the model, with more general initial conditions that allow the support of the solutions to intersect the turning curve, demonstrate the strong sensitivity of the evacuation time to the same cost parameter; this instability arises from interactions between agents and the turning curve.}

	\vspace{4pt}
	
	
	
	\vspace{4pt}
\end{abstract}

 \maketitle
	
\section{Introduction}

The Hughes model~\cite{Hughes02} for crowd evacuation attracted considerable attention in the mathematical and in the engineering communities in the past twenty years. Even its simplest one-dimensional version {poses significant challenges for in-depth analysis}. For the sake of more precise modeling and/or to facilitate the rigorous analysis of models' features, many variants of the original Hughes model were proposed. 
We refer to~\cite{survey2023} for a survey focused on the one-dimensional case. 
While regularized models inspired by the setting~\cite{Hughes02} (Hughes-kind models with a regularized directions field for agents' evacuation) can be throughly treated mathematically in one or more space dimensions, {well-posedness results for the original model, which includes a discontinuity in the direction field, are rare. Our study focuses on this non-regularized setting.}

The one-dimensional Hughes' model is a macroscopic pedestrian flow model that {couples two variants of the LWR model (by Lighthill-Whitham~\cite{LWR1} and Richards~\cite{LWR2}). Specifically, the model uses an LWR model with negative flux on the left of the coupling and an LWR model with positive flux on the right. As a result, it is represented by a scalar conservation law with discontinuous flux, determining the two possible directions of the agents towards the two exits.} This PDE governs the evolution of the density $\rho$ of agents. The location $x=\xi(t)$ of the jump discontinuity of the directions field is known as the ``turning curve''; the value $\xi(t)$ is obtained from the solution of {an} eikonal equation involving a density-dependent cost function and the density distribution $\rho(t,\cdot\,)$.

{The only available existence results in the $\mathbf{BV}$ setting require initial data carefully selected to exclude any interactions of the agents with the turning curve (see~\cite{AmadoriGoatinRosini} and~\cite{DiFrancescoFagioliRosiniRussoKRM}). This restriction is straightforward to interpret in microscopic Follow-the-Leader models for Hughes' dynamics, as developed in~\cite{ARS2023, DiFrancescoFagioliRosiniRussoKRM, DiFrancescoFagioliRosiniRusso}. 
One of our goals is to provide a complete and rigorous uniqueness analysis for such solutions. We establish this result under the assumption of affine cost functions, which, while restrictive, allows us to derive precise stability estimates.}

Recently, existence for general data was proved {in the case of an affine cost function~\cite{AndrGirard2024, ARS2023}}; only the $ \mathbf{BV}_{\textbf{loc}}$ regularity away from the turning curve is known for these solutions. {The affine cost assumption ensures that the turning curve $\xi$ has the $W^{1,\infty}$ regularity, which is an essential ingredient of the arguments of both works \cite{AndrGirard2024} and \cite{ARS2023}.} Moreover, a convergent approximation scheme, which can be seen as a microscopic {Hughes model,} was put forward in~\cite{ARS2023}. Although limited to the affine cost, these works address the situation where decoupling of the two equations (the one for the density $\rho$ and the one for the turning curve $\xi$) corresponds to two well-posed problems, see~\cite{AndrGirard2024} for details and a fixed-point argument based on the decoupling. This advantageous feature is due to the Lipschitz continuity of the turning curve, which is a consequence of the choice of affine costs. 
{In contrast, a recent work~\cite{Storbugt2024} established a very general existence result, for a cost given as the sum of a convex and a concave functions, in the original setting where the turning curve belongs to $\mathbf{BV} \cap \C0$. However, this setting does not resolve the question of the uniqueness of the density $\rho$, given the turning curve $\xi$.}

A first goal of the present contribution is to complement the existence result of~\cite{AndrGirard2024, ARS2023} for the one-dimensional Hughes model with affine cost {$c:=1+\alpha\,\rho$} by a discussion of the issue of stability of solutions, with respect to the {initial} data and to the cost parameter {$\alpha$}; this includes uniqueness (for the coupled problem). In a sense, we go into the direction opposite of~\cite{Storbugt2024}, where the notion of solution was relaxed enough to ensure existence, at the price of leaving open the uniqueness issue even for subproblems resulting from decoupling of the equations. 

{By strengthening the notion of a solution and restricting the set of possible initial data, we focus on a well-defined subset of Hughes dynamics, ensuring} uniqueness and several stability features.
On this track, we obtain a uniqueness and continuous dependence result in the $\L1$-norm, adapting the machinery developed for conservation laws with discontinuous flux.
Our analysis applies under the assumption of ``well-separated solutions'', that is, in absence of agents in a vicinity of the turning curve.
{It is worth noting that an affine cost function is not essential to ensure the existence of well-separated solutions. In fact, this can be established with any cost function $c$ that satisfies the following conditions: $c \in \C2$, $c' \geq 0$, $c'' > 0$, and $c(0) = 1$, as shown in \cite{AmadoriGoatinRosini, DiFrancescoFagioliRosiniRussoKRM}. However, in this work, we specifically assume that the cost function is affine to derive stability estimates.}
This study highlights the outreach and the limitations of the established {discontinuous-flux} stability techniques (cf.~e.g.~\cite{AndrLagoTakaSeguin14, SyllaNHM}) for handling the Hughes model.
 
A second objective of the present contribution is to exploit the approximation scheme of~\cite{ARS2023} (in its fully discrete version) in order to study numerically the dependence of the evacuation time on initial data and on the {cost parameter}, in solutions that are not necessarily well-separated. We point out the lack of stability {of the evacuation time}, which is attributed to interactions between agents and the turning curve.
{We stress that, from an engineering viewpoint, the evacuation time for a given datum is an essential information about the solution. Note, however, that the observed instabilities in the evacuation time do not imply instability of the solution in the $\L1$-norm, which is the most natural way to measure the proximity of two solutions of conservation laws. Further research appears necessary to investigate the $\L1$ (in)stability in solutions that are not well-separated.}

The paper is organized as follows. In Section~\ref{sec:summary} we {start by stating} the one-dimensional Hughes problem, recall the definition of solutions and the existence results available in the literature for the case of {affine} costs. We also recall the partial existence and uniqueness results known for more general costs. In Section~\ref{s:ContDep} we develop the study of uniqueness and stability of solutions, in the $\L1$-norm for the pedestrian density, under a strong structural assumption on the solutions that we call ``well-separation''; the statements are provided within the section. In Section~\ref{sec:scheme-and-numerics} we recall the semi-discrete deterministic microscopic approximation of the problem and use the scheme to simulate a variety of scenarios, focusing on the dependence of the {evacuation} time upon the cost parameter and perturbation on the initial datum.

\section{The model, the definitions and some results recalled}
\label{sec:summary}

Consider a bounded corridor, parametrized by $x \in \mathfrak{C} := (-1,1)$, where the endpoints $x = \pm1$ model the two exits.
The agents' density $\rho=\rho(t,x)\geq0$ in the corridor is bounded from above by the maximal density $\rho_{\max}>0$. 
The Hughes model~\cite{Hughes02} in this one-dimensional situation consists of the PDE
\begin{equation}\label{eq:reformul}
	\rho_{t}+\bigl(V(t,x) \, {f( \rho )}\bigr)_{x} = 0,
\end{equation}
(with appropriate boundary conditions, discussed later), where for every fixed $t\geq 0$, the velocity field $V(t,x)=\frac{\partial_x \phi(t,x)}{|\partial_x \phi(t,x)|}$ is determined from the solution $\phi(t,\cdot\,)$ of the eikonal equation 
\begin{align*}
&|\partial_x \phi(t,x)|=c(\rho(t,x)) \text{ in $\mathfrak{C}$,}&
&\phi(t,\cdot\,)|_{\partial \mathfrak{C}}=0.
\end{align*}
Here, $c\in \C0([0,1]; [1,\infty))$ is called ``the cost function'' and is given as a part of the model.

Following~\cite{AmadoriDiFrancesco, AmadoriGoatinRosini, DiFrancescoFagioliRosiniRussoKRM, ElKhatibGoatinRosini}, due to the explicit resolution of the eikonal equation, we can adapt the ``turning-curve approach'', expressing $V(t,x)$ as $\sign\left(x-\xi(t)\right)$ where, given $t\geq 0$, the value $\xi(t)\in \mathfrak{C}$ is computed by solving the scalar equation
\begin{equation}
\label{e:cost0}
\int_{-1}^{\xi(t)}c\bigl( \rho(t,y)\bigr) \, {\rm{d}} y = \int_{\xi(t)}^{1}c\bigl( \rho(t,y)\bigr) \, {\rm{d}} y.
\end{equation}
The map $t\mapsto \xi(t)$ is called the turning curve throughout the paper; indeed, it can be shown (cf.~\cite{Storbugt2024}) that even in a very general setting, $\xi$ is a $\C0$ map.

Further, as detailed in~\cite{ARS2023}, one can replace the standard homogeneous Dirichlet boundary condition for the Hughes model by an ``open-end formulation'' in which the conservation law resolution is extended from $x\in \mathfrak{C}$ to $x\in \mathbb{R}$. {This modification merely requires extending the initial density datum $\overline{\rho}$, originally defined on $\mathfrak{C}$, to a function defined on $\mathbb{R}$ and supported in $\mathfrak{C}$.} Somewhat abusively, we will keep the notation $\overline{\rho}$ for the initial density datum extended by zero to $\mathbb{R}\setminus\mathfrak{C}$.
{We recall this result in Proposition~\ref{prop:open-end}, as it will be instrumental in Section~\ref{s:ContDep}, where we utilize the Dirichlet formulation.}

To sum up, we adopt a formulation from previous works on the subject. This formulation, which is equivalent to the original one for an arbitrary cost function, involves solving in the appropriate entropy sense the Cauchy problem for a discontinuous-flux scalar conservation law
\begin{equation}\label{eq:reformul}
\left\{
\begin{array}{@{}l@{\quad}l@{\ }l@{}}
\rho_{t}+\bigl(\sign\left(x-\xi(t)\right) {f( \rho )}\bigr)_{x} = 0,& t > 0, &x \in \mathbb{R},
\\[3pt]
\rho(0,x)=\overline{\rho}(x),&& x\in \mathbb{R}.
\end{array}	
\right.
\end{equation}
Additionally, for $t\geq 0$, $\xi(t)$ is determined by equation \eqref{e:cost0}.

For the sake of notational convenience, {we introduce the Kruzhkov entropy fluxes}
\begin{align}
\label{e:Phi}
\Phi(t,x, \rho ,\kappa,\xi) &:= \sign\bigl(x-\xi(t)\bigr) \, \sign( \rho -\kappa) \, \bigl( f( \rho ) - f(\kappa) \bigr).
\end{align}
All the above discussion is then summarized in the following solution notion.
\begin{defn}[Open-end formulation]
	\label{d:entro-bis}
	Consider a measurable initial datum $\overline{\rho} \colon \mathbb{R} \to [0, \rho_{\max}]$ supported in $\mathfrak{C}$.
{A couple
\[(\rho,\xi) \in \left( \Lloc1\left([0, \infty) \times \mathbb{R};[0, \rho_{\max}]\right) \cap \C0([0,\infty);\L1(\mathbb{R})) \right) \times \mathbf{Lip}([0, \infty);\mathfrak{C})\]
is an entropy solution of the one-dimensional Hughes model if it satisfies \eqref{e:cost0} for a.e.\ $t>0$, the initial condition in \eqref{eq:reformul} for a.e.\ $x\in\mathbb{R}$, and the PDE in \eqref{eq:reformul} in the sense that it satisfies the entropy inequalities
	\begin{align}
		\label{e:entro-bis}
		0 \leq{}& \int_0^{ \infty} \int_\mathbb{R} \bigl( | \rho -\kappa| \, \varphi_t + \Phi(t,x, \rho ,\kappa,\xi) \, \varphi_x \bigr) \, {\rm{d}} x \, {\rm{d}} t
		+2 \int_0^{ \infty} f(\kappa) \, \varphi\bigl(t,\xi(t)\bigr) \, {\rm{d}} t
	\end{align}
	for all $\kappa \in [0, \rho_{\max}]$ and all test functions $\varphi \in \Cc \infty\left((0, \infty) \times \mathbb{R};[0, \infty)\right)$.}
	
	Additionally, we say that the solution is $ \mathbf{BV}$-regular if for all $T>0$, {$\sup \{\TV(\rho(t,\cdot\,)) : t\in[0,T]\}$} is finite.
\end{defn}
{The former integral in \eqref{e:entro-bis} originates from the classical Kruzhkov \cite[Definition~1]{Kruzhkov}. The latter integral in \eqref{e:entro-bis} accounts for the possible arise of non-classical shocks at the turning point.}
The above definition can be slightly simplified, dropping the entropy admissibility conditions on $x=\xi(t)$.
Indeed, in the Hughes context the mere conservation, expressed through the classical Rankine-Hugoniot condition, is enough to encode the solution admissibility at the turning curve; we refer to \cite[Definition~1.6]{AndrGirard2024} for details.

The following existence result is proven in~\cite{ARS2023} using a constructive approach, which is meaningful for applications, and {employing an abstract fixed-point argument under weaker assumptions} in~\cite{AndrGirard2024}, see \cite[Theorem~5]{ARS2023} and \cite[Proposition~1.9]{AndrGirard2024}{, respectively.}

\begin{thm}[Existence for {affine} costs]
	\label{t:existence} 
	Consider an affine cost $c(\rho):=1+\alpha\,\rho$ with $\alpha\geq 0$.
	{Assume that $f$ satisfies the following conditions:}
	\begin{itemize}
	\item
	$f\in \mathbf{W^{\pmb{1,\infty}}}((0,\rho_{\max});[0,\infty))$ is concave and such that $f(0) = 0 = f(\rho_{\max})$.
	\item
	The set $\{\rho \in [0,\rho_{\max}] : f'(\rho)=0\}$ has zero measure.	
	\end{itemize}
	Then for any initial datum $\overline{\rho} \in \L\infty(\mathbb{R};[0,\rho_{\max}])$ such that $\overline{\rho}(x)=0$ for $x \in \mathbb{R} \setminus \mathfrak{C}$, there exists an entropy solution $(\rho,\xi)$ of the one-dimensional Hughes model.
	
	{If furthermore $f\in\C2([0,\rho_{\max}];[0,\infty))$ is such that $f(\rho) > \rho \, f'(\rho)$ and $f(\rho) - \rho \, f'(\rho) + \rho^{2} \, f''(\rho) \leq0$ for any $\rho\in(0,\rho_{\max}]$,} then an entropy solution can be constructed as the limit of the many-particle Hughes model detailed in \cite[Section~3]{ARS2023}.
\end{thm}

{Let us recall the physical interpretation of the parameter $\alpha \geq 0$. When $\alpha = 0$, pedestrians choose the nearest exit, ignoring the overall density of the crowd. This behaviour is characteristic of panic scenarios. Conversely, as $\alpha$ increases, the model places greater importance on distributing pedestrians more evenly between the two exits. As a result, higher values of $\alpha$ discourage movement toward overcrowded regions.}
Note that it is an open question whether the solutions of the one-dimensional Hughes model, even for the case of {affine} costs, are $ \mathbf{BV}$-regular for $ \mathbf{BV}$ initial data.

\section{Stability for well-separated solutions for {affine} costs}\label{s:ContDep}
Now we define the specific class of solutions for which we analyze in detail the uniqueness and stability properties.
\begin{defn}[Well-separated solution]
	\label{d:well-separated}
	An entropy solution $(\rho,\xi)$ of the one-dimensional Hughes model is called well-separated if the limits $\rho(t,\xi(t)^\pm)$ are equal to zero {for a.e. $t>0$.}
\end{defn}
We recall that the existence of strong one-sided traces of $\rho$ on the Lipschitz curve $x=\xi(t)$, exploited in the above definition, is automatic: it follows from the local entropy inequalities and the strict non-linearity of the flux, see~\cite{MR1869441}.

Notice that well-separated solutions were investigated in~\cite{AmadoriGoatinRosini, DiFrancescoFagioliRosiniRussoKRM} (see the latter work for an interpretation of the results in terms of absence of crossing {between agents' paths} and the trajectory of the turning curve).
We then assert the following stability claim.
\begin{thm}[Stability of well-separated solutions]
	\label{t:WS-stability} 
	Consider an affine cost $c(\rho):=1+\alpha \, \rho$ with $\alpha\geq 0$.
	For $i\in\{1,2\}$, let $( \rho_{i},\xi_{i})$ be a well-separated entropy solution of the one-dimensional Hughes model{, in the sense of Definitions~\ref{d:entro-bis},~\ref{d:well-separated},} corresponding to the initial datum $\overline{\rho}_{i}$ and the cost parameter $\alpha_{i}$. 
	If $\overline{\rho}_1$ is a $\mathbf{BV}$ function, then for a.e.\ $t>0$ there exists a constant $C$ that depends, in a non-decreasing way, only on $t$ and $\max\{\alpha_{1},\alpha_{2}\}$, such that
	\begin{equation}\label{eq:stability}
		\left\| \rho_{1}(t,\cdot\,)- \rho_{2}(t,\cdot\,)\right\|_{\L1(\mathfrak{C})} \leq C\bigl(|\alpha_{1}-\alpha_{2}|+\left\|\overline{\rho}_{1}-\overline{\rho}_{2}\right\|_{\L1(\mathfrak{C})}\bigr).
	\end{equation}
	In particular, for any fixed cost parameter $\alpha$ and initial datum $\overline{\rho}\in  \mathbf{BV}(\mathbb{R})$ there exists at most one well-separated solution.
\end{thm}

We emphasize that the $\mathbf{BV}$-regularity of {$\overline{\rho}_{1}$} is a crucial ingredient of the stability argument; indeed, being understood that $\xi_{1}\not\equiv \xi_{2}$, we have to call upon quantitative stability of solutions of conservation laws $\rho_t + f(t,x, \rho )_x=0$ with respect to perturbations of the flux $f$. 
All such results (see~\cite{BouchutPerthame, KarlsenRisebro, Mercier} for some well known ones) rely upon the $\mathbf{BV}$ in space estimate on at least one of the two solutions.
Moreover, the open-end reformulation of Definition~\ref{d:entro-bis} is instrumental in our stability analysis. 
Finally, we stress that the well-separation assumption is not merely {technical: dropping} it would require new ideas for the sake of uniqueness analysis.

We turn to the proof of Theorem~\ref{t:WS-stability}.
The claim of the theorem follows by Lemma~\ref{lem:BV-reg} highlighting the fact that the well-separated solutions corresponding to $\mathbf{BV}$ data are automatically $ \mathbf{BV}$-regular, and by a straightforward combination of two distinct ingredients provided in Propositions~\ref{prop:xidot} and~\ref{prop:rho-of-xi}. We state the Lemma, and both Propositions before turning to {their proofs.}
\begin{lem}\label{lem:BV-reg}
	Let $(\rho,\xi)$ be a well-separated solution of the one-dimensional Hughes model, in the sense of Definitions~\ref{d:entro-bis},~\ref{d:well-separated}.
	If the initial datum $\overline{\rho}$ is in $\mathbf{BV}(\mathbb{R})$, then the solution is $\mathbf{BV}$-regular, i.e., $\rho\in \Lloc\infty([0,\infty);\mathbf{BV}(\mathbb{R}))$.
\end{lem}

\begin{prop}\label{prop:xidot}
{Using the same notation as in Theorem~\ref{t:WS-stability}, there exist $\tau>0$ and $K>0$, which} depend only on $\max\{\alpha_{1},\alpha_{2}\}$, such that
\begin{align}\label{eq:integral-dotxi-estim}
\int_0^\tau|\dot{\xi}_{1}(s)-\dot{\xi}_{2}(s)|\, {\rm{d}} s &\leq K\Bigl(|\alpha_{1}-\alpha_{2}|+\left\|\overline{\rho}_{1}-\overline{\rho}_{2}\right\|_{\L1(\mathfrak{C})}\Bigr),
\\\label{eq:xi(0)-estim}
|\xi_{1}(0)-\xi_{2}(0)| &\leq K\Bigl(|\alpha_{1}-\alpha_{2}|
+\left\|\overline{\rho}_{1}-\overline{\rho}_{2}\right\|_{\L1(\mathfrak{C})}\Bigr).
\end{align}
\end{prop}

\begin{prop}\label{prop:rho-of-xi}
{Using the same notation as in Theorem~\ref{t:WS-stability}, if} $\overline{\rho}_1$ is a $ \mathbf{BV}$ function, then for a.e.\ $t>0$ there holds
\begin{align}\nonumber
&\left\| \rho_{1}(t,\cdot\,)- \rho_{2}(t,\cdot\,)\right\|_{\L1(\mathfrak{C})} 
\\\leq{}&\left\|\overline{\rho}_{1}-\overline{\rho}_{2}\right\|_{\L1(\mathfrak{C})}
+2 \, C_1(t) \left(|\xi_{1}(0)-\xi_{2}(0)|+ \int_0^t |\dot{\xi}_{1}(s)-\dot{\xi}_{2}(s)| \, {\rm{d}} s \right),
\label{eq:rho-estim}
\end{align}
where {$C_1(t) := \sup \left\{\TV\left( \rho_{1}(s,\cdot\,)\right) : s \in[0,t]\right\}$.}
\end{prop}	

{The proof of Theorem~\ref{t:WS-stability} uses a stop-and-restart procedure, which requires the following result, adapted from \cite{ARS2023}, along with its straightforward corollary. To streamline the claims in the subsequent two statements, whenever we refer to a ``solution'' of a discontinuous-flux scalar conservation law in the spatial domain $\mathfrak{C}$ or $\mathbb{R}$, it is understood in the sense of the entropy inequalities \eqref{d:entro-bis} for non-negative test functions $\varphi \in \Cc\infty$, supported in the open space-time domain where the PDE is considered.
\begin{prop}{(adapted from \cite[Prop.~31]{ARS2023})}
\label{prop:open-end}
Let $\xi\in  \mathbf{Lip}([0, \infty);\mathfrak{C})$ be given.
For a function $\overline{\rho} \colon \mathfrak{C} \to [0, \rho_{\max}]$, we identify $\overline{\rho}$ with its extension to $\mathbb{R}$ by assigning it the value zero on $\mathbb{R}\setminus \mathfrak{C}$. 
Then the following problems are equivalent:
\begin{equation}
	\label{e:IBVP}
	\left\{\begin{array}{@{}l@{\qquad}l@{\ }l@{}}
		\rho_{t}+\bigl(\sign\left(x-\xi(t)\right) f( \rho )\bigr)_{x} = 0, &x\in \mathfrak{C},&t>0,\\
		\rho (0,x) = \overline{\rho}(x), &x\in \mathfrak{C},\\
		\rho(t,\pm 1)=0, &&t>0,
	\end{array}\right.
\end{equation}
where the boundary condition is understood in the Bardos-LeRoux-N\'ed\'elec (BLN) sense (see \cite{AndrSbihi, MR542510}), and
\begin{equation}
	\label{e:IVP}
	\left\{\begin{array}{@{}l@{\qquad}l@{\ }l@{}}
		\rho_{t}+\bigl(\sign\left(x-\xi(t)\right) \, f( \rho )\bigr)_{x} = 0, &x\in \mathbb{R},&t>0,\\
		\rho (0,x) = \overline{\rho}(x), &x\in \mathbb{R}.
	\end{array}\right.
\end{equation}
More precisely, given $\overline{\rho}$, any solution to  Problem~\eqref{e:IBVP} can be seen as the restriction to $x\in \mathfrak{C}$ of a solution to Problem~\eqref{e:IVP}; reciprocally, 
any solution to Problem~\eqref{e:IVP} can be seen as an extension to $x\in \mathbb{R}$ of a solution to Problem~\eqref{e:IBVP}.
\end{prop}
Now, let $\mathbbm{1}_{\mathfrak{C}}$ be the characteristic function of $\mathfrak{C}$, i.e., the function equal to $1$ on $\mathfrak{C}$ and equal to $0$ elsewhere. 

\begin{cor}\label{cor:restart}
Given a solution $(\rho,\xi)$ of the one-dimensional Hughes model in the sense of Definition~\ref{d:entro-bis}, and given $t_0>0$, the restriction $\rho|_{[t_0,\infty)\times\mathfrak{C}}$ coincides with the function $(t,x)\mapsto r(t-t_0,x)$, where $r$ is the\footnote{{Given a Lipschitz continuous turning curve, the solution $r$ of \eqref{eq:t=t0} is unique, see \cite{AndrGirard2024}.}} solution of the problem
\begin{equation}\label{eq:t=t0}
\left\{\begin{array}{@{}l@{\quad}l@{\ }l@{}}
r_{t}+\bigl(\sign\left(x-\xi(t+t_0)\right) f( r )\bigr)_{x} = 0,& t > 0, &x \in \mathbb{R},
\\[3pt]
r(0,x)={\rho}(t_0,x)\mathbbm{1}_{\mathfrak{C}}(x),&& x\in \mathbb{R}.
\end{array}\right.
\end{equation}
\end{cor}
\begin{proof}
By the time continuity result of~\cite{CancesGallouet} for local entropy solutions of scalar conservation laws, it follows that $\rho_{1}$, $\rho_{2}$ are continuous in time with values in $\L1(\mathfrak{C})$ (cf.~\cite{SyllaNHM, Towers-BV} for details, in similar situations).
Then it is obvious that the solution of \eqref{e:IBVP} fulfils the flow property, i.e., the value of the solution at any point $t_0>0$ can be taken as the initial datum for solving the equation on $[t_0,\infty)$, and such resolution provides the same density $\rho$ in $[t_0,\infty)\times \mathfrak{C}$. Then we apply Proposition~\ref{prop:open-end} with $\xi(\cdot+t_0)$ in the place of $\xi(\cdot)$ and the initial datum $\bar\rho(\cdot):={\rho}(t_0,\cdot\,)\mathbbm{1}_{\mathfrak{C}}(\cdot)$.  
\end{proof}}
\begin{proof}[Proof of Theorem~\ref{t:WS-stability}]
The result for $t \in [0,\tau]$ follows by simply plugging the bounds \eqref{eq:integral-dotxi-estim}, \eqref{eq:xi(0)-estim} of Proposition~\ref{prop:xidot} into estimate \eqref{eq:rho-estim} of Proposition~\ref{prop:rho-of-xi}, having in mind Lemma~\ref{lem:BV-reg} that ensures the requested $ \mathbf{BV}$ bound on $\rho_1$.

In order to achieve the global in time result, recall that $\tau$ {in \eqref{eq:integral-dotxi-estim}} depends only on $\max\{\alpha_{1},\alpha_{2}\}$. 
Then we bootstrap the argument, by means of the stop-and-restart procedure, taking $t=k\,\tau$ for initial time, $k \in\mathbb{N}$. 
{Indeed, by Corollary~\ref{cor:restart}, we know that given $\xi_{1,2}$, the functions $\rho_{1,2}|_{\mathfrak{C}}(t,\cdot\,)$ for $t\geq k \, \tau$ can be obtained by translating by $t_0=k\,\tau$ in time and restricting to $x\in\mathfrak{C}$ the solutions of \eqref{eq:t=t0} (with the given turning curves $\xi_{1,2}$, respectively)  and with initial data  $\rho_{1,2}|_{\mathfrak{C}}(k \, \tau,\cdot\,)$ extended by zero to $\mathbb{R}\setminus\mathfrak{C}$.
Applying the argument on $[0,\tau]$ to the time-translated problem \eqref{eq:t=t0}, we have the estimate}
\begin{align*}
\left\| \rho_{1}(t,\cdot\,)- \rho_{2}(t,\cdot\,)\right\|_{\L1(\mathfrak{C})} \leq K\left(|\alpha_{1}-\alpha_{2}| +\left\| \rho_{1}(k \, \tau,\cdot\,)- \rho_{2}(k \, \tau,\cdot\,)\right\|_{\L1(\mathfrak{C})}\right)
\\
\forall t \in[k \, \tau,(k+1) \, \tau]&,
\end{align*}
where we can recursively plug the estimate of $\left\| \rho_{1}(k \, \tau,\cdot\,) - \rho_{2}(k \, \tau,\cdot\,) \right\|_{\L1(\mathfrak{C})}$.
Thus we extend the control of $\left\| \rho_{1}(t,\cdot\,)- \rho_{2}(t,\cdot\,)\right\|_{\L1(\mathfrak{C})}$ of the form \eqref{eq:stability} to any time $t>0$, with the constant $C$ that grows with $t$ at most like $\max\{1,K\}^{m+1}$, being {$m:=\lfloor t/\tau\rfloor$.}
\end{proof}	 

\begin{proof}[Proof of Proposition~\ref{prop:rho-of-xi}]
The proof is analogous to the arguments that can be found in~\cite{AndrLagoTakaSeguin14, DelleMonacheGoatin17, SyllaNHM}, regarding continuous dependence of solutions of discontinuous-flux conservation laws with respect to a moving interface. 
For application of the argument to the Hughes model at hand, it is important to rely upon the ``open-end formulation'' of Definition~\ref{d:entro-bis}, because changes of variable leaving invariant the half-plane $t>0,\,x\in\mathbb{R}$ (but not leaving invariant the band $t>0,\,x\in \mathfrak{C}$) will be instrumental.

In the subdomains $\{(t,x) : \pm (x-\xi_{1}(t))>0\}$ the function $\rho_{1}$ solves the conservation laws $\partial_t \rho_{1}+\partial_x(\pm f( \rho_{1}))=0$ in the standard Kruzhkov entropy sense. 
It is readily checked, by a change of variable in the entropy formulation (written with Lipschitz continuous test functions), that in the subdomains $\Theta_\pm:=\{(t,x) : \pm(x-\xi_{2}(t))>0\}$ the translated function $\rho_{3}(t,x):= \rho_{1}(t,x-\xi_{2}(t)+\xi_{1}(t))$ solves the scalar conservation law with translated flux
\begin{equation}\label{eq:auxil-equ}
\partial_t \rho_{3} + \partial_x\bigl( \pm f( \rho_{3}) + (\dot{\xi}_{2}(t) - \dot{\xi}_{1}(t)) \rho_{3} \bigr) = 0
\end{equation}
in the standard Kruzhkov entropy sense. 
Recall that $\rho_{2}$ solves the analogous problem with the fluxes $\pm f( \rho_{2})$ in the same subdomains $\Theta_\pm$. 
We then use the standard estimate of continuous dependence on the flux within the doubling of variables argument~\cite{BouchutPerthame, KarlsenRisebro, Mercier}: for all smooth non-negative $\varphi$ compactly supported in $\Theta_- \cup \Theta_+$, we have
\begin{align}\nonumber
&- \int_0^{\infty} \int_\mathbb{R} \Bigl(
\begin{aligned}[t]
&\varphi_t \left| \rho_{3}(t,x)- \rho_{2} (t,x)\right|
\\{}+{}& \varphi_x \, \sign\bigl(x-\xi_{2}(t)\bigr) \, \sign\bigl( \rho_{3}(t,x)- \rho_{2}(t,x)\bigr)\\
&\times\bigl(f\bigl( \rho_{3}(t,x)\bigr)- f\bigl( \rho_{2}(t,x)\bigr)\bigr)\Bigr) {\rm{d}} x\, {\rm{d}} t
\end{aligned}\\\nonumber
\leq{}& \int_\mathbb{R} \varphi(0,x) \, |\overline{\rho}_{3}(x)-\overline{\rho}_{2}(x)| \, {\rm{d}}x
\\&+ \int_0^{\infty} \left\|\varphi(t,\cdot\,)\right\|_{\L\infty} \, |\dot{\xi}_{2}(t)-\dot{\xi}_{1}(t)| \, \TV\bigl(\rho_{3}(t,\cdot\,)\bigr)\, {\rm{d}} t,
\label{eq:ContDepLocal}
\end{align} 	
where $\overline{\rho}_{3}(x):=\overline{\rho}_{1}(x-\xi_{2}(0)+\xi_{1}(0))$ is the initial condition for the solution $\rho_{3}$ of \eqref{eq:auxil-equ}.

Now, note that for well-separated solutions, it is straightforward to drop the assumption that $\varphi$ is zero in a neighbourhood of the curve $\{(t,x) : x=\xi_{2}(t)\}$. 
Indeed, {take a general} $\varphi \in \Cc{ \infty}([0, \infty) \times \mathbb{R})$ and consider truncated test functions $\varphi_m(t,x) := \varphi(t,x) \, \eta(m \, (x-\xi_{2}^{m}(t))) \in \Cc{ \infty}(\Theta_- \cup \Theta_+) $. 
Here $\xi_{2}^{m}$ is a smooth approximation of $\xi_{2}$ such that $\dot{\xi}_{2}$ is uniformly bounded and $\left\|\xi_{2}-\xi_{2}^{m}\right\|_{\L\infty} \leq 1/m$ (such approximation is constructed by convolution); $\eta \in\C\infty(\mathbb{R})$ is even, $\eta'(z) \geq 0$ for $z>0$, $\eta(z)=1$ for $|z|\geq 1$, and $\eta\equiv 0$ in a neighbourhood of $z=0$. 
Upon substituting $\varphi_m$ into \eqref{eq:ContDepLocal}, the integrals of the terms
\begin{gather*}
m \, \eta'\bigl(m (x-\xi_{2}^{m}(t))\bigr) \, \dot{\xi}_{2}^{m}(t) \left| \rho_{3}- \rho_{2}\right| \varphi,\\
m \, \eta'\bigl(m (x-\xi_{2}^{m}(t))\bigr) \sign\bigl( \rho_{3}(t,x)- \rho_{2}(t,x)\bigr) \, \bigl(f\bigl( \rho_{3}(t,x)\bigr) - f\bigl( \rho_{2}(t,x)\bigr)\bigr) \, \varphi
\end{gather*}
vanish as $m\to \infty$ due to the assumption that $\rho_{1}$, $\rho_{2}$ are well-separated, since it means that the traces of {$\rho_{3}$ and $\rho_{2}$ along the curve $x = \xi_{2}(t)$ are equal to zero almost everywhere.} 
Indeed, one can perform the change of variables $y := m\,(x-\xi_{2}(t))$ in the integrals of these terms. 
The dominated convergence can be applied since $\rho_{2,3}(t,\xi_{2}(t)+\frac{y}{m})\to 0$ pointwise, as $m\to \infty$, while the support of $y\mapsto \eta'\bigl(y+m\,(\xi_{2}(t)-\xi_{2}^{m}(t))\bigr)$ lies within the fixed interval $[-2,2]$, by the choice of $\eta$ and of $\xi_{2}^{m}$.

Then, as in the standard Kruzhkov $\L1$-contraction argument, we can let $\varphi$ converge to the indicator function of $[0,t) \times \mathbb{R}$ and infer
\begin{equation}\label{eq:ContDepGlobal}
\int_\mathbb{R} | \rho_{3}(t,x)- \rho_{2} (t,x)|\, {\rm{d}} x 
\leq \int_\mathbb{R} |\overline{\rho}_{3}(x)-\overline{\rho}_{2}(x)|\, {\rm{d}} x+ C_{1}(t) \int_0^{t} |\dot{\xi}_{1}(s)-\dot{\xi}_{2}(s)|\, {\rm{d}} s,
\end{equation} 
where {$C_1(t) := \sup \left\{\TV\left( \rho_{1}(s,\cdot\,)\right) : s \in[0,t]\right\}$} and we also used the fact that, by construction, $\TV( \rho_{3}(s,\cdot\,))=\TV( \rho_{1}(s,\cdot\,))$.
From the definition of $\rho_{3}$, we also infer
\begin{align*}
\int_\mathbb{R} |\overline{\rho}_{1}(x)-\overline{\rho}_{3}(x)|\, {\rm{d}} x 
& \leq C_{1}(0) \, |\xi_{1}(0)-\xi_{2}(0)|
\leq C_{1}(t) \, |\xi_{1}(0)-\xi_{2}(0)| ,\\
\int_\mathbb{R} | \rho_{1}(t,x)- \rho_{3} (t,x)|\, {\rm{d}} x 
& \leq C_{1}(t) \, |\xi_{1}(t)-\xi_{2}(t)| \\
& \leq C_{1}(t) \left(|\xi_{1}(0)-\xi_{2}(0)|+ \int_0^t |\dot{\xi}_{1}(s)-\dot{\xi}_{2}(s)|\, {\rm{d}} s\right).
\end{align*}
Assembling these bounds with \eqref{eq:ContDepGlobal} via the triangle inequality, we infer the claim of the proposition.
\end{proof}
\begin{proof}[Proof of Proposition~\ref{prop:xidot}]
We start by proving \eqref{eq:xi(0)-estim}. 
It follows from \eqref{e:cost0} that 
\begin{align*}
&\int_{-1}^{\xi_{2}(0)} c_2(\overline{\rho}_{2}(x))\, {\rm{d}} x - \int_{-1}^{\xi_{1}(0)} c_{1}(\overline{\rho}_{1}(x))\, {\rm{d}} x 
\\={}& 
\int^1_{\xi_{2}(0)} c_2(\overline{\rho}_{2}(x))\, {\rm{d}} x - \int^1_{\xi_{1}(0)} c_{1}(\overline{\rho}_{1}(x))\, {\rm{d}} x,
\end{align*}
which can be rewritten as	
\[2 \int_{\xi_{1}(0)}^{\xi_{2}(0)} c_2(\overline{\rho}_{2}(x))\, {\rm{d}} x 
= \int^1_{-1} \sign(x-\xi_{1}(0)) \bigl(c_2(\overline{\rho}_{2}(x))-c_{1}(\overline{\rho}_{1}(x)) \bigr)\, {\rm{d}} x.\]
The choice of $c_i=1+\alpha_i \, \rho$, $i\in\{1,2\}$, now yields
\[2 \, |\xi_{2}(0)-\xi_{1}(0)| \leq \int_{-1}^1 |\alpha_{2} \, \overline{\rho}_{2}(x) - \alpha_{1} \, \overline{\rho}_{1}(x)|\, {\rm{d}} x,\]
which readily leads to \eqref{eq:xi(0)-estim} having in mind that $\left\|\overline{\rho}_{2}\right\|_{\L1(\mathfrak C)} \leq 2 \rho_{\max}$.

\smallskip
Now, let us admit for a while the explicit expression for $\dot{\xi}_{1}$, which can be obtained by a formal calculation. 
Keeping in mind that $c_{1}( \rho_{1}(t,\xi_{1}(t)^\pm))=c_{1}(0)=1$ due to the assumption that $\rho_{1}$ is a well-separated solution, substituting $\partial_x(\pm f( \rho_{1}))$ in the place of $\partial_t \rho_{1}$ for $\pm (x-\xi_{1}(t))>0$, we exhibit the formula
\begin{align}\label{eq:dotxi}
&\dot{\xi}_{1}(t)
= -\frac{\alpha_1}{2} \int_{-1}^1 f( \rho_{1}(t,x))_x \, {\rm{d}} x
= \frac{\alpha_1}{2} \Bigl(f\bigl( \rho_{1}(t,-1^+)\bigr) - f\bigl( \rho_{1}(t,1^-)\bigr)\Bigr)&
&
\end{align}
for a.e.\ $t>0$. 
The demonstration of the claim \eqref{eq:dotxi} is postponed to the end of the proof.

We now exploit the expression \eqref{eq:dotxi} for $\dot{\xi}_{1}$ (and the analogous expression for $\dot{\xi}_{2}$) in order to reach to \eqref{eq:integral-dotxi-estim}, for appropriately defined $\tau>0$ and $K>0$. 
Set
\begin{align*}
&\Lambda:=\max\{ \|f'\|_{\L\infty}, \|\dot{\xi}_{1}\|_{\L\infty}, \|\dot{\xi}_{2}\|_{\L\infty}\},\\
&\begin{aligned}
x_*&:=\min\{\xi_{1}(0),\xi_{2}(0)\},&
x^*&:=\max\{\xi_{1}(0),\xi_{2}(0)\},\\
\tau^*&:=\frac{1-x^*}{\Lambda},&
\tau_*&:=\frac{x_*+1}{\Lambda},&
\tau&:=\min\{\tau^*,\tau_*\}.
\end{aligned}
\end{align*}
Because the costs $c_{1}$, $c_{2}$ take values in $[1,1+\max\{\alpha_{1},\alpha_{2}\} \, \rho_{\max}]$, it is easily seen from \eqref{e:cost0} that $\xi_{1}(0)$, $\xi_{2}(0)$ belong to $[-1+\delta,1-\delta]$ with $\delta=\bigl(1+\frac{1}{2} \, \max\{\alpha_{1},\alpha_{2}\} \, \rho_{\max}\bigr)^{-1}>0$ that only depends on $\max\{\alpha_{1},\alpha_{2}\}$ and on the fixed value $\rho_{\max}$; moreover, in view of \eqref{eq:dotxi}, $\|\dot{\xi}_{1}\|_{\L\infty}$, $\|\dot{\xi}_{2}\|_{\L\infty}$ are bounded by a constant times $\max\{\alpha_{1},\alpha_{2}\}$.
Therefore $\tau>0$ depends only on $\max\{\alpha_{1},\alpha_{2}\}$ (it also depends on $\rho_{\max}$ and other fixed parameters of the model).

Now, let $\mathcal{T}^*$ be the interior of the triangle with vertices $(x^*,0)$, $(1,0)$ and $(1,\tau^*)$.
By the definition of $x^*$ and of $\Lambda$, both $\rho_{1}$ and $\rho_{2}$ verify in $\mathcal{T}^*$ the same homogeneous scalar conservation law with flux $f${, which} is non-affine on any interval.
In this situation, strong traces of $\rho_{1}$, $\rho_{2}$ as $t\to 0^+$ (the initial trace) and as $x\to 1^-$ (the boundary trace) exist, see~\cite{Panov_traces1,Panov_traces2}. 
It follows that, first, the so-called Kato inequality in $\mathcal{T}^*$ is fulfilled:
\begin{align*}
- \iint_{\mathcal{T}^*} \Bigl(| \rho_{1}- \rho_{2}|\varphi_t+ \sign( \rho_{1}- \rho_{2})\bigl(f( \rho_{1})-f( \rho_{2})\bigr)\varphi_x \Bigr)\, {\rm{d}} x \, {\rm{d}} t \leq 0
\\
\forall \varphi \in \Cc\infty(\mathcal{T}^*)&{}.
\end{align*}
Second, one can proceed as in the classical setting of Kruzhkov~\cite{Kruzhkov}, approximating the characteristic function of $\mathcal{T}^*$ by a sequence of $\varphi \in \C \infty_0(\mathcal{T}^*)$; note that, like in~\cite{Kruzhkov}, we have chosen the slope $\Lambda$ of the oblique part of the boundary of $\mathcal{T}^*$ larger than $\|f'\|_{\L\infty}$. 
We find 
\begin{align*}
&\int_{0}^{\tau^*} \sign\bigl( \rho_{1}(t,1^-)- \rho_{2}(t,1^-)\bigr) \left( f( \rho_{1}(t,1^-))-f( \rho_{2}(t,1^-))\right)\, {\rm{d}} t
\\\leq{}&
\int_{x^*}^1 \left|\overline{\rho}_{1}(x)-\overline{\rho}_{2}(x)\right|\, {\rm{d}} x.
\end{align*}
Finally, recalling the BLN interpretation~\cite{MR542510} of the Dirichlet boundary condition (see also~\cite{AndrSbihi}, where the boundary condition is interpreted in terms of monotone subgraphs of the graph of $f$), we point out that
\begin{align*}
&\sign\bigl( \rho_{1}(t,1^-)- \rho_{2}(t,1^-)\bigr) \left(f( \rho_{1}(t,1^-))-f( \rho_{2}(t,1^-))\right)
\\={}& \left|f( \rho_{1}(t,1^-))-f( \rho_{2}(t,1^-))\right|.
\end{align*}
To sum up, we find
\[\int_{0}^{\tau^*} \left|f( \rho_{1}(t,1^-))-f( \rho_{2}(t,1^-))\right|\, {\rm{d}} t \leq \left\|\overline{\rho}_{1}-\overline{\rho}_{2}\right\|_{\L1([x^*,1])}.\]
Further, the same inequality holds with $1^-$, $\tau^*$ and $[x^*,1]$ replaced by $-1^+$, $\tau_*$ and $[-1,x_*]$, respectively.

Recalling \eqref{eq:dotxi}, which we also write for $\dot{\xi}_{2}$ with $\alpha_{2}=(\alpha_{2}-\alpha_{1})+\alpha_{1}$, we finally deduce \eqref{eq:integral-dotxi-estim} under the precise form
\[\int_0^\tau|\dot{\xi}_{1}(s)-\dot{\xi}_{2}(s)|\, {\rm{d}} s \leq \tau\left\|f\right\|_{\L\infty}|\alpha_{1}-\alpha_{2}|+
\frac{\alpha_{1}}{2} \left\|\overline{\rho}_{1}-\overline{\rho}_{2}\right\|_{\L1(\mathfrak{C})}.\]

To conclude the proof, we now turn to the justification of \eqref{eq:dotxi}. 
Recall that by definition, $\xi_{1}$ is Lipschitz continuous, therefore its derivative $\dot{\xi}_{1}$ is defined a.e., and it is enough to establish
\begin{align}\nonumber
2 \int_0^{\infty} \theta(t) \, \dot{\xi}_{1}(t)\, {\rm{d}} t 
= \alpha_1 \int_0^{\infty} \theta(t) \, \bigl(f\bigl( \rho_{1}(t,-1^+)\bigr) - f\bigl( \rho_{1}(t,1^-)\bigr)\bigr) \, {\rm{d}} t
\\\forall \theta \in \Cc{ \infty}((0,\infty))&{}.
\label{eq:xidot-weak}
\end{align}
As a starting point, let us multiply \eqref{e:cost0} by $\dot\theta$ and integrate in time. 
This leads to
\begin{align}\nonumber
0&= \int_0^{\infty} \dot\theta(t) \int_{\mathfrak{C}} \sign\left(x-\xi_{1}(t)\right) \, c_{1}( \rho_{1}(t,x))\, {\rm{d}} x\, {\rm{d}} t
\\\nonumber
&= 
\begin{aligned}[t]
&\int_0^{\infty} \int_{\mathfrak{C}} \dot\theta(t) \, \sign\left(x-\xi_{1}(t)\right)\, {\rm{d}} x \, {\rm{d}} t
\\&
+ \alpha_{1} \int_0^{\infty} \int_{\mathfrak{C}} \dot\theta(t) \, \sign\left(x-\xi_{1}(t)\right) \, \rho_{1}(t,x)\, {\rm{d}} x \, {\rm{d}} t
\end{aligned}
\\\label{eq:I1+Irho}
&=: I_1+\alpha_{1} \, I_ \rho . 
\end{align}
We then consider
a sequence $\{\xi_{1}^m\}_m$ of $\C \infty$ approximations of $\xi_{1}$, which converge uniformly on $[0,\infty)$ while keeping uniformly bounded derivatives $\dot{\xi}_{1}^m$ converging pointwise to $\dot{\xi}_{1}$ (such approximations can be obtained by regularizing $\dot{\xi}_{1}$ by convolution). 
Without loss of generality, we can assume $\left\|\xi_{1}^m-\xi_{1}\right\|_{\L\infty} \leq 1/m$. 
Further, we approximate $\sign(\cdot)$ by functions $\eta_m(z):=\eta(mz)$, where $\eta \in\C\infty(\mathbb{R})$ is non-decreasing and $\eta(z)=\sign(z)$ for $|z|\geq 1$.
Note that $\eta_m(x-\xi_{1}^m(t))\to \sign(x-\xi_{1}(t))$ for all $(t,x)$ such that $x\neq \xi_{1}(t)$.

The dominated convergence readily yields $I_1=\lim_{m\to\infty} I_1^m$, where $I_1^m$ is defined by replacing $\sign(x-\xi_{1}(t))$ by $\eta_m(x-\xi_{1}^m(t))$ in the definition of $I_1$. 
Now we integrate by parts in $t$ in the expression of $I_1^m$. 
For sufficiently large $m$, using the fact that $\xi_{1}$ takes values in $[-1+\delta,1-\delta]$ with some $\delta>0$, as stated here above, we find
\begin{align}\nonumber
I_1=\lim_{m\to \infty} I_1^m
&=\lim_{m\to \infty} \int_0^{\infty} \theta(t) \, \dot{\xi}_{1}^m(t) \int_{\mathfrak{C}} \eta_m'(x-\xi_{1}^m(t))\, {\rm{d}} x\, {\rm{d}} t
\\&=2 \int_0^{\infty} \theta(t) \, \dot{\xi}_{1}(t)\, {\rm{d}} t.
\label{eq:I1}
\end{align}
In order to calculate $I_ \rho $, let us point out that the weak formulation contained in \eqref{e:entro-bis} implies
\begin{equation}\label{eq:weak-formul}
\int_0^{ \infty} \int_{\mathfrak{C}} \bigl( \rho_{1}\, \varphi_t + \sign\left(x-\xi_{1}(t)\right) f( \rho_{1})\, \varphi_x \bigr) \, {\rm{d}} x \, {\rm{d}} t=0
\end{equation}
for all $\varphi \in \Cc \infty\left((0,\infty) \times \mathfrak{C}\right)$.
We then consider in \eqref{eq:weak-formul} the test functions $\varphi_m(t,x)=\theta(t) \, \psi(x) \, \eta_m(x-\xi_{1}^m(t))$ with $\psi \in \Cc{\infty}(\mathfrak{C})$ and $\theta$, $\eta_m$, $\xi_{1}^m$ defined above. 
The choices we made for $\eta_m$ and $\xi_{1}^m$ ensure that the factor $|\eta'(m(x-\xi_{1}^m(t)))|$ has its support included in the set $\{(t,x) : |x-\xi_{1}(t)| \leq \frac 2m\}$; note that this factor is $\L \infty$ bounded uniformly in $m$.
It follows that, as $m\to \infty$, the integrals of
\begin{align*}
&\theta(t) \, \psi(x) \, m \, \eta'\bigl(m\left(x-\xi_{1}^m(t)\right)\bigr) \, \dot{\xi}_{1}^m(t)\, \rho_{1},&
&\theta(t)\,\psi(x)\,m\,\eta'\bigl(m\left(x-\xi_{1}^m(t)\right)\bigr)\,f( \rho_{1})
\end{align*}
vanish due to the assumption of zero traces of $\rho_{1}$ on the curve $x=\xi_{1}(t)$.
This is justified via dominated convergence argument in transformed variables $(y,t)$, $y:=m \, (x-\xi_{1}(t))$, having in mind the above remark on the support of the $\eta'$ factor and the zero trace assumption meaning that $\rho_{1}(t,\xi_{1}(t)+\frac ym)\to 0$ pointwise, as $m\to \infty$.

As $\sign(x-\xi_{1}(t)) \, \eta\bigl(m(x-\xi_{1}(t))\bigr)\to 1$ a.e., with another application of the dominated convergence theorem we infer 
\begin{equation*}
\int_0^{\infty} \int_{\mathfrak{C}} \left( \dot\theta(t) \, \psi(x) \, \sign\left(x-\xi_{1}(t)\right) \, \rho_{1}\, + \theta(t) \,\psi'(x) \, f( \rho_{1})\, \right) \, {\rm{d}} x \, {\rm{d}} t=0.
\end{equation*}
Finally, to reach to $I_ \rho $ we let $\psi$ converge to the characteristic function of $\mathfrak{C}$; since strong boundary traces $\rho_{1}(\cdot,\pm1^\mp)$ exist (see~\cite{Panov_traces2}), we get
\begin{equation}\label{eq:Irho}
I_ \rho = \int_0^{\infty} \theta(t) \left(f( \rho_{1}(t,1^-)) - f( \rho_{1}(t,-1^+))\right) \, {\rm{d}} t.
\end{equation}
Assembling \eqref{eq:I1}, \eqref{eq:Irho} within \eqref{eq:I1+Irho}, we reach to \eqref{eq:xidot-weak} and conclude the proof.
\end{proof}
\begin{proof}[Proof of Lemma~\ref{lem:BV-reg}]
As in the proof of Proposition~\ref{prop:rho-of-xi}, we exploit the open-end formulation and make the Lipschitz change of variables rectifying the turning curve. Given a well-separated solution $(\rho,\xi)$ of the Hughes model, we first consider its restriction on the subdomain $\{(t,x) : x-\xi(t))>0\}$. We consider
\[\tilde \rho^+ \colon (0,\infty)\times (0,\infty) \ni (t,x)\mapsto \rho(t,x-\xi(t)).\]
On its domain of definition, this function fulfils in the Kruzhkov entropy sense the conservation law
\begin{equation}\label{eq:auxil-equ-bis}
\partial_t \tilde\rho^+ + \partial_x\bigl(f( \tilde\rho^+) - \dot{\xi}(t)\tilde \rho^+ \bigr) = 0,
\end{equation}
moreover, it fulfils in the sense of strong trace the initial condition $\tilde\rho^+(0,x)=\bar \rho(x)$ (for $x\in (0,\infty)$) and, due to the well-separation assumption, it fulfils the boundary condition $\tilde\rho^+(t,0^+)=0$ for $t\in (0,\infty)$. 
The function $(t,y)\mapsto \tilde\rho^+(t,y)$ is the unique entropy solution of the above conservation law with the zero Dirichlet boundary condition at $y=0^+$ taken in the {BLN} sense~\cite{MR542510} (see also~\cite{AndrSbihi, MR1869441}); we stress that the $\mathbf{BV}$ regularity need not be used in the definition, see~\cite{ColomboRossi-ProcEdinb2019}. 
The uniqueness of such solution is proved in~\cite{ColomboRossi-ProcEdinb2019} under the $\C1$ regularity assumption on the time-dependence of the flux, but the proof works in the case of the flux $\rho\mapsto f(\rho)-\dot\xi \, \rho$ with the mere $\L\infty$ coefficient $\dot\xi$ (in particular, the doubling of variables argument supports irregular time-dependence of the flux, see~\cite{MR2568808}). 
Due to this uniqueness claim and to the standard passage to the limit arguments, the  $\mathbf{BV}$ bound
{\[\sup \{\TV(\tilde\rho^+(t,\cdot\,)) : t\in[0,T]\} \leq \TV(\bar\rho|_{\mathbb{R}_+})\]}
is inherited from the analogous bound on the approximate solutions (e.g., it is enough to approximate the time-dependence of the flux in a piecewise constant way; one can also consider a full discretization by a finite volume method or by wave-front tracking, cf.~\cite{ColomboRossi-ProcEdinb2019}). 
In the same way, considering $\tilde \rho^-: (0,\infty)\times (-\infty,0) \ni (t,x)\mapsto \rho(t,x-\xi(t)),$ we get the bound 
{\[\sup \{\TV(\tilde\rho^-(t,\cdot\,)) : t\in[0,T]\}\leq \TV(\bar\rho|_{\mathbb{R}_-}).\]}
By construction, for all $t>0$ there holds 
\[ \TV(\tilde\rho(t,\cdot\,))=\TV(\tilde\rho^-(t,\cdot\,))+\TV(\tilde\rho^+(t,\cdot\,))\]
(for each function $\rho,\tilde \rho^+, \tilde \rho^-$ the variation is taken over its spatial domain). 
Whence the result follows.
\end{proof}	

\section{Numerical experiments}
\label{sec:scheme-and-numerics}

We now investigate numerically some properties of the solution to equation \eqref{eq:reformul}, with particular attention to the relations between the parameter $\alpha$ in the affine cost and the evacuation time. The numerical scheme used in the following is inspired by the deterministic \emph{Follow-the-Leader} approximation introduced in~\cite{ARS2023, DiFrancescoFagioliRosiniRussoKRM}. 
{For notation simplicity, we denote $\disint{\mu,\nu} := [\mu,\nu] \cap \Z$ for any $\mu,\nu\in\Z$ with $\mu \leq \nu$.}

{Let $v\in\C2([0,\rho_{\max}];[0,\infty))$ be such that $v_{\max} := v(0) > 0$, $v(\rho_{\max})=0$ and $f(\rho) = \rho \, v(\rho)$.
Assume that $v'(\rho)<0$ and $v'(\rho)+\rho\,v''(\rho)\leq0$ for any $\rho\in[0,\rho_{\max}]$.}
Let $\bar{\rho}$ be in $\L\infty(\mathfrak{C};[0,\rho_{\max}])$. For a fixed integer $n\in \mathbb{N}$, set $N := 2^n$ and $m := 2^{-n} \, M$, where $M:=\|\bar{\rho}\|_{\L1}$.
Denote
\[
\bar{x}_0 := \min\left\{\mbox{supp}(\bar{\rho})\right\}
\]
and recursively define
\begin{align*}
&\bar{x}_i := \inf\left\{ x > \bar{x}_{i-1} \colon \int_{\bar{x}_{i-1}}^x \bar{\rho}(y) \,{\rm{d}} y \geq m\right\},&
&i\in\disint{1,N}.
\end{align*}
The above equation defines the set of $N + 1$ particles' initial positions $-1 \leq \bar{x}_0 < \bar{x}_1 < \ldots < \bar{x}_{N-1} < \bar{x}_N \leq 1$, with the property that the mass of the density $\bar{\rho}$ in each interval $(\bar{x}_i,\bar{x}_{i+1})$ is exactly $m$. We then let the particles evolve according to 
\begin{equation}\label{e:FtL}
\begin{cases}
\dot{x}_0(t)= -v_{\max},\\
\dot{x}_i(t)=-v\left(\frac{m}{x_{i}(t)-x_{i-1}(t)}\right), &i \in \disint{1, I_0},\\
\dot{x}_i(t)=v\left(\frac{m}{x_{i+1}(t)-x_i(t)}\right), &i \in \disint{I_0+1, N-1},\\
\dot{x}_N(t)= v_{\max},\\
x_i(0)=\bar{x}_i,&i \in \disint{0, N},
\end{cases}
\end{equation}
where the index $I_0$ is selected as
\begin{align*}
 I_0:=\max_{{i\in\disint{0,N}}}\Bigl\{\frac{2}{\alpha m} \, x_i(t)< &\#\left\{{j \in \disint{0,N}} \,:\,x_i(t)<x_j(t)<1\right\}\\
& -\#\left\{{j \in \disint{0,N}} \,:\,-1<x_j(t)<x_i(t)\right\} \Bigr\}.
\end{align*}
We consider the corresponding local discrete densities
\begin{align*}
&R_{i+1/2}(t) := \frac{m}{x_{i+1}(t)-x_i(t)},&
&i \in \disint{0,N-1}\setminus\left\{I_0\right\},
\\
&R_{i+1/2}(t) := 0,&
&i \in \left\{-1,I_0,N\right\},
\end{align*}
and the corresponding piecewise constant discrete density $\rho^n : [0,\infty)\times\mathbb{R} \to [0,\rho_{\max}]$ defined by
\begin{equation}\label{e:approsolDPA}
\rho^n(t,x) := \sum_{i=0}^{N-1} R_{i+1/2}(t) \, \mathbbm{1}_{[x_i(t),x_{i+1}(t))}(x).
\end{equation}
The approximate turning point $\xi^n(t)$ is implicitly uniquely defined by
\begin{equation}\label{e:xiDPA}
\int_{-1}^{\xi^n(t)} c\left(\rho^n(t,y)\right) \,{\rm{d}} y = \int^1_{\xi^n(t)} c\left(\rho^n(t,y)\right) \,{\rm{d}} y.
\end{equation}
Clearly $\xi^n(t)$ belongs to $\mathfrak{C}$ for any $t\geq 0$. 

In the simulation we consider the classical choice for the velocity $v$
\begin{equation*}
  v(\rho) := v_{\max}\left(1-\frac{\rho}{\rho_{\max}}\right),
\end{equation*}
and we fix $v_{\max} = \rho_{\max} =1$. With $n=500$ particles we solve system \eqref{e:FtL} by the MATLAB solver ode23s with time step { $\Delta t=0.004$}. We define the microscopic evacuation time as
\begin{equation*}
 T_{\rm mic} = \min_{t\in [0,T]}\bigl\{ \#\left\{{i \in\disint{0,N}} \,:\,x_i(t)\in \mathfrak{C}\bigr\} = 0\right\}.
\end{equation*}

\subsubsection*{{First numerical simulation}}

\begin{figure}[!htb]
\includegraphics[width=.49\linewidth, trim={11mm 4mm 17mm 10mm}, clip]{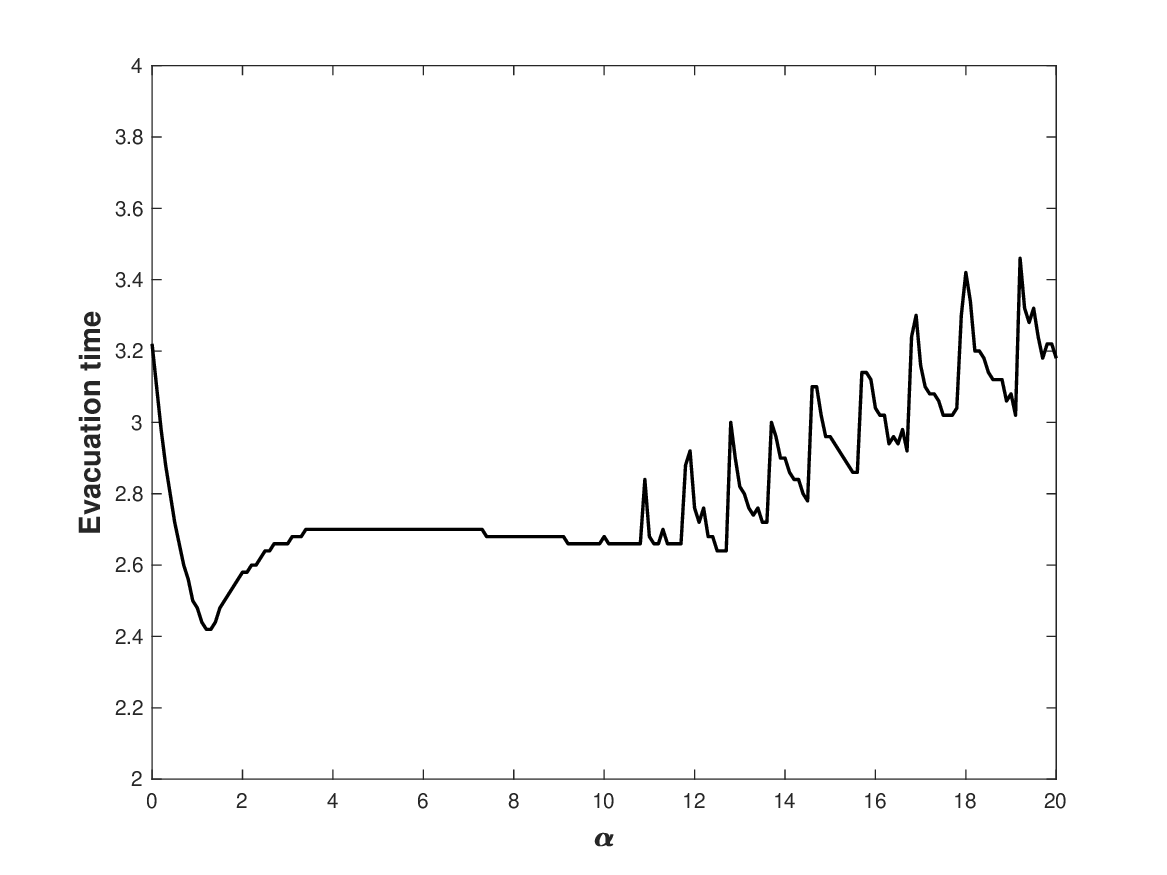}
\hfill
\includegraphics[width=.49\linewidth, trim={11mm 4mm 17mm 10mm}, clip]{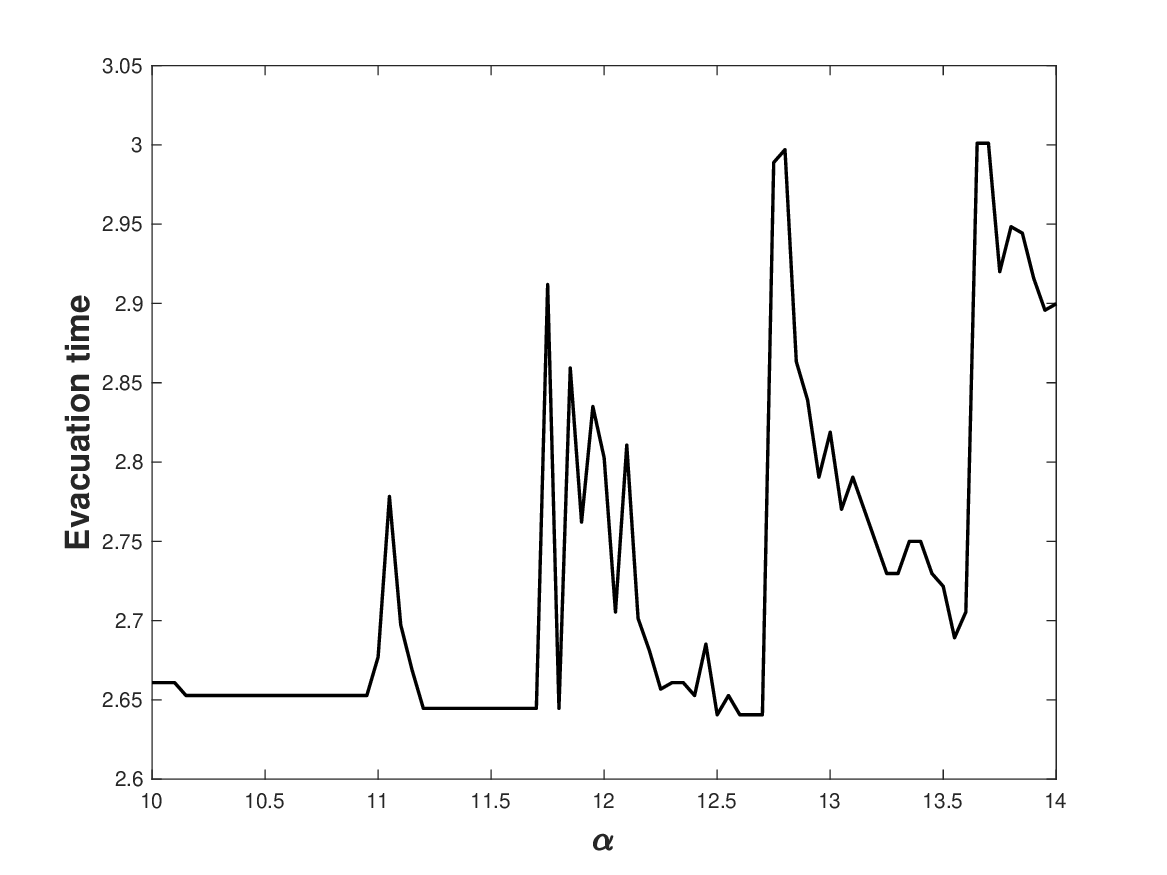}
\caption{$T_{\rm mic}$ as function of $\alpha$ with initial datum \eqref{num_in_con}.
Left: $\alpha \in [0,20]$, step size $\Delta \alpha=0.1$.
Right: $\alpha \in [10,14]$, step size $\Delta \alpha=0.05$.}
\label{fig:alpha}
\end{figure}

Under the initial condition 
\begin{equation}\label{num_in_con}
\bar{\rho}(x) =\begin{cases} 0.9 & \mbox{ if } -1\leq x < -0.5,\\
                          0.9 & \mbox{ if } -0.4\leq x < 0,\\
                          0 & \mbox{ otherwise},
\end{cases}
\end{equation}
we obtain the evolution of $T_{\rm mic}$ as a function of $\alpha$, see \figurename~\ref{fig:alpha}. { We remark that \figurename~\ref{fig:alpha}, left, shows a similar overall behaviour compared to \cite[\figurename~6]{ARS2023}, using the same grid for $\alpha$ with a uniform step size $\Delta \alpha = 0.1$. Note that \figurename~\ref{fig:alpha}, left, exhibits more oscillations compared to \cite[\figurename~6]{ARS2023}, where a fully explicit scheme is used with the same time step. Moreover, more oscillations appear when the step size is reduced to $\Delta \alpha = 0.05$, as shown in \figurename~\ref{fig:alpha}, right.} 

\subsubsection*{{Second numerical simulation}}

\begin{figure}[!htb]
\includegraphics[width=.49\linewidth, trim={11mm 4mm 17mm 4mm}, clip]{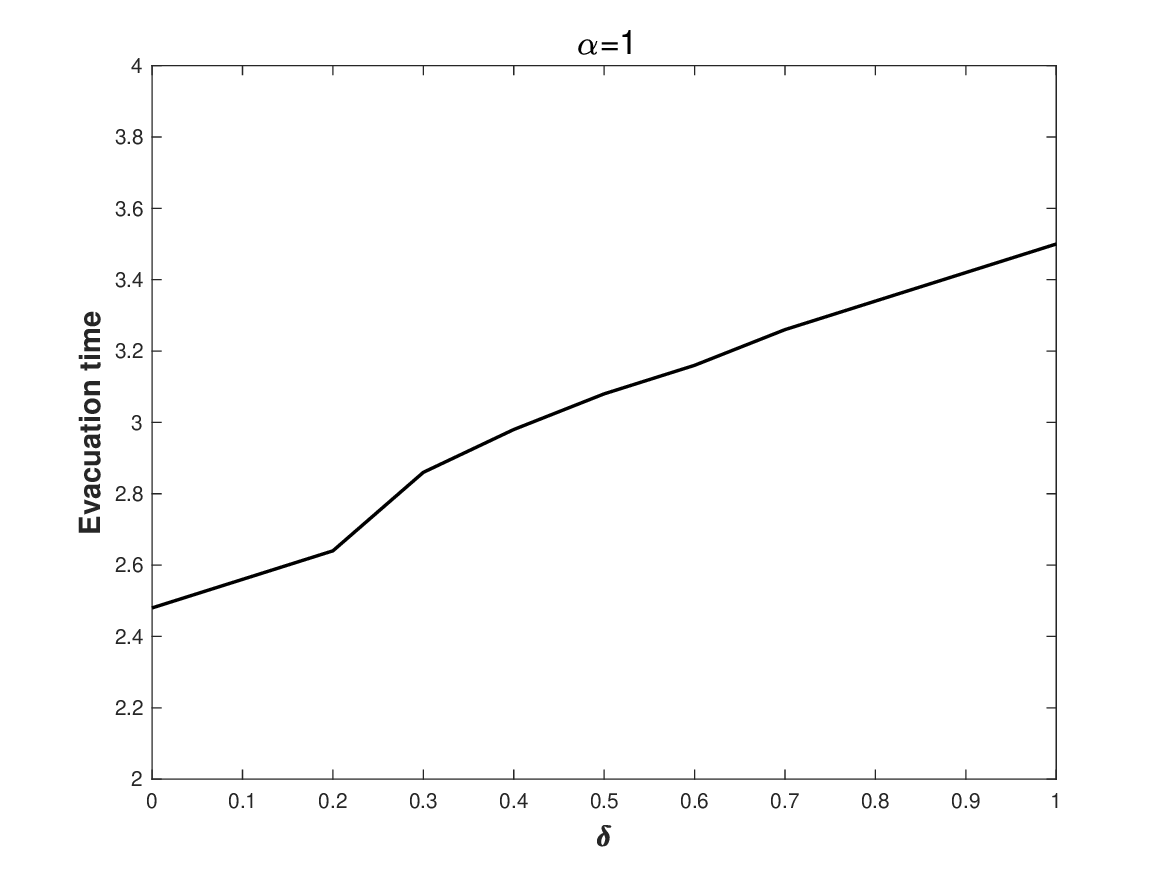}
\hfill
\includegraphics[width=.49\linewidth, trim={11mm 4mm 17mm 4mm}, clip]{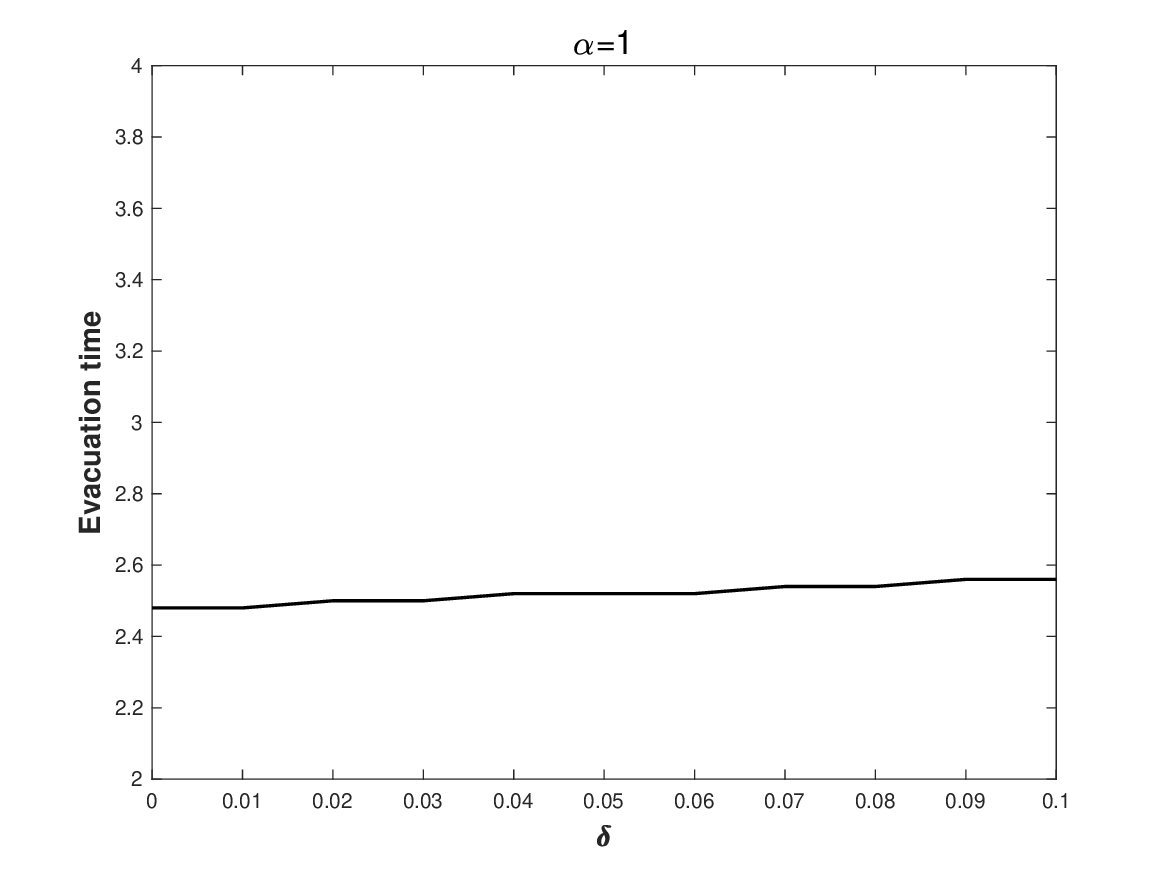}\\
\includegraphics[width=.49\linewidth, trim={11mm 4mm 17mm 4mm}, clip]{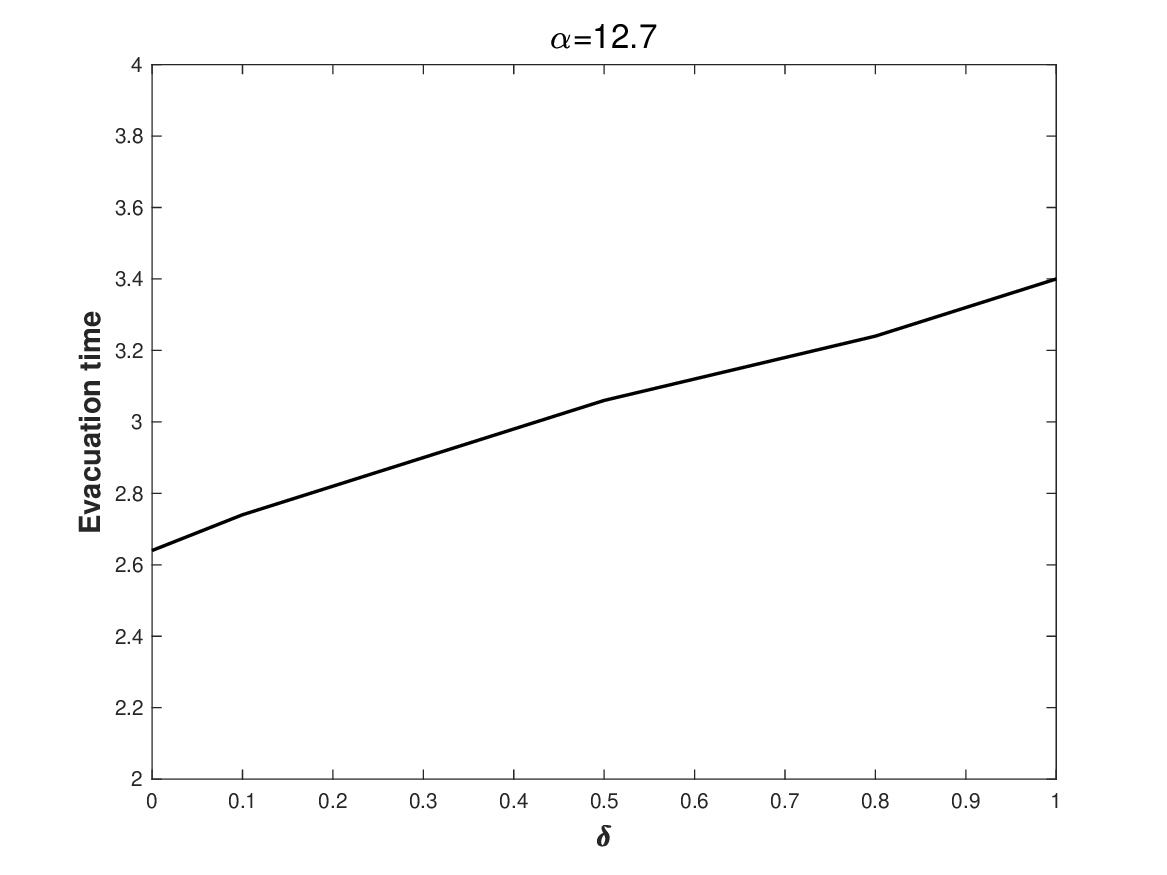}
\hfill
\includegraphics[width=.49\linewidth, trim={11mm 4mm 17mm 4mm}, clip]{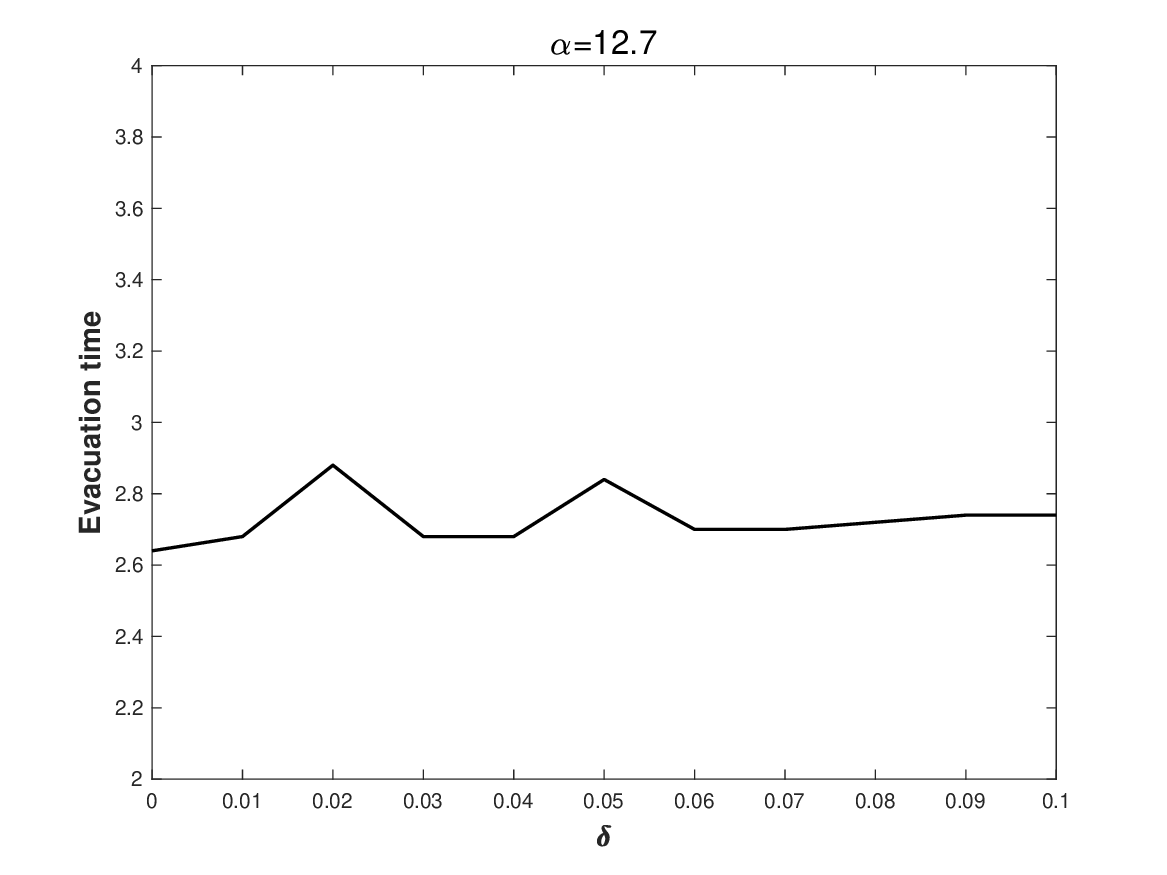}
\caption{$T_{\rm mic}$ as function of $\delta$ with $\Delta \delta = 0.01$ and initial datum \eqref{num_in_con_2}.
Left: $\delta \in [0,1]$.
Right: $\delta \in [0,0.1]$.
Top: $\alpha=1$.
Bottom: $\alpha=12.7$.}
\label{fig:delta}
\end{figure}

We then consider the variation { of $T_{\rm mic}$ corresponding to a shift perturbation of the initial datum \eqref{num_in_con}.
We consider the initial datum}
\begin{equation}\label{num_in_con_2}
\bar{\rho}(x) =\begin{cases} 0.9 & \mbox{ if } -1+\delta\leq x < -0.5+\delta,\\
                          0.9 & \mbox{ if } -0.4+\delta\leq x < 0+\delta,\\
                          0 & \mbox{ otherwise},
\end{cases}
\end{equation}
for $\delta\in[0,1]$. In \figurename~\ref{fig:delta} we plot the evacuation time in the cases $\alpha=1$ and $\alpha=12.7$, the latter corresponding to one of the jumps in \figurename~\ref{fig:alpha}, right. 
In both cases we observe a continuous dependence of $T_{\rm mic}$ with respect to $\delta$, with possible jumps in its derivatives, see \figurename~\ref{fig:delta}, right.

\subsubsection*{{Third numerical simulation}}

\begin{figure}[!htb]
\includegraphics[width=.49\linewidth, trim={10mm 4mm 16mm 4mm}, clip]{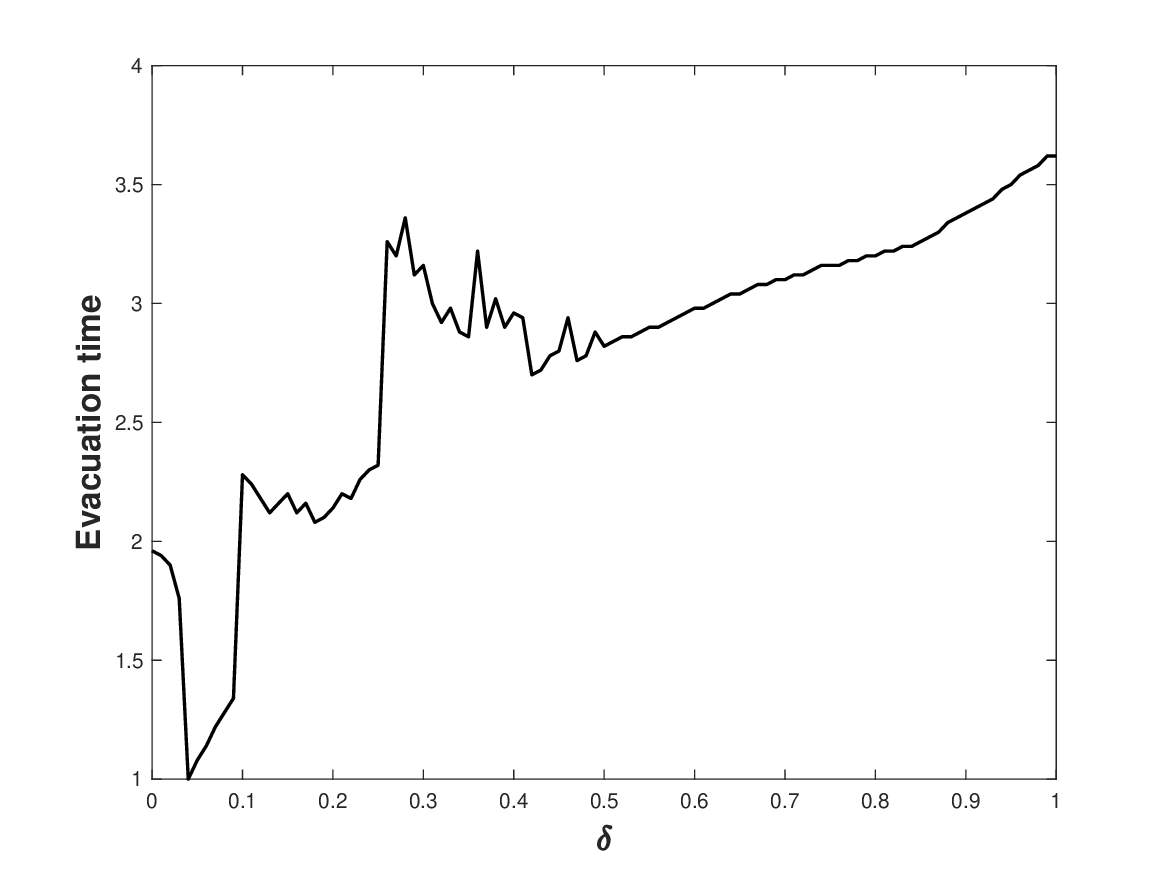}
\caption{$T_{\rm mic}$ as function of $\delta$  with $\Delta \delta = 0.01$, initial datum \eqref{num_in_con_4} and $\alpha=12.7$.}
\label{fig:epsilon}
\end{figure}

{
The last example concerns the initial condition
\begin{equation}\label{num_in_con_4}
\bar{\rho}(x) =\begin{cases} 0.9 & \mbox{ if } -0.5-\delta\leq x < -0.5+\delta,\\
                          0.9 & \mbox{ if } 0.75\leq x < 1,\\
                          0 & \mbox{ otherwise},
\end{cases}
\end{equation}
for $\delta\in[0,1]$.
We observe in \figurename~\ref{fig:epsilon} the discontinuities of $T_{\rm mic}$. Such discontinuities correspond to new crossing of the turning curves by particles trajectories, see \figurename~\ref{fig:crossing} where we plot particles path for values of $\delta$ across the two jumps at $\delta=0.1$ and $\delta=0.26$.} 

\begin{figure}[!htb]
\includegraphics[scale=0.25]{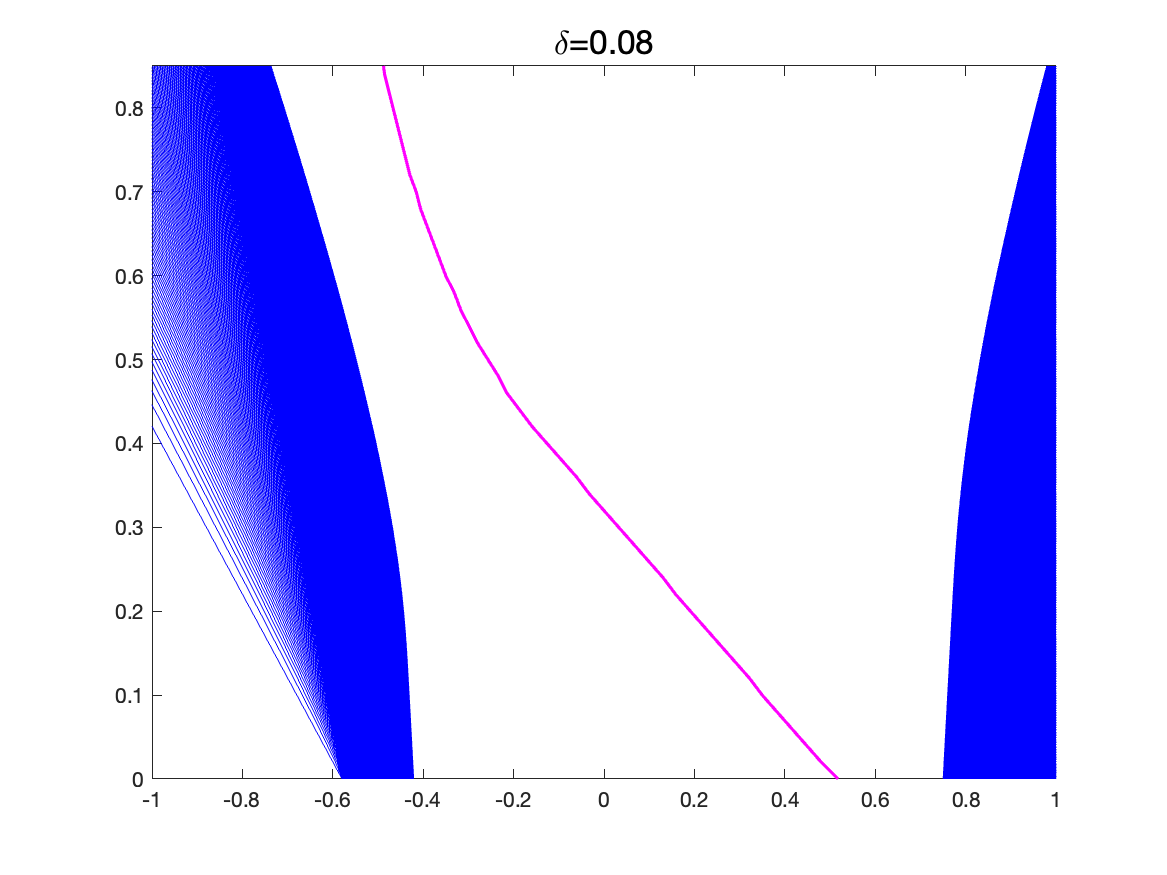}
\includegraphics[scale=0.25]{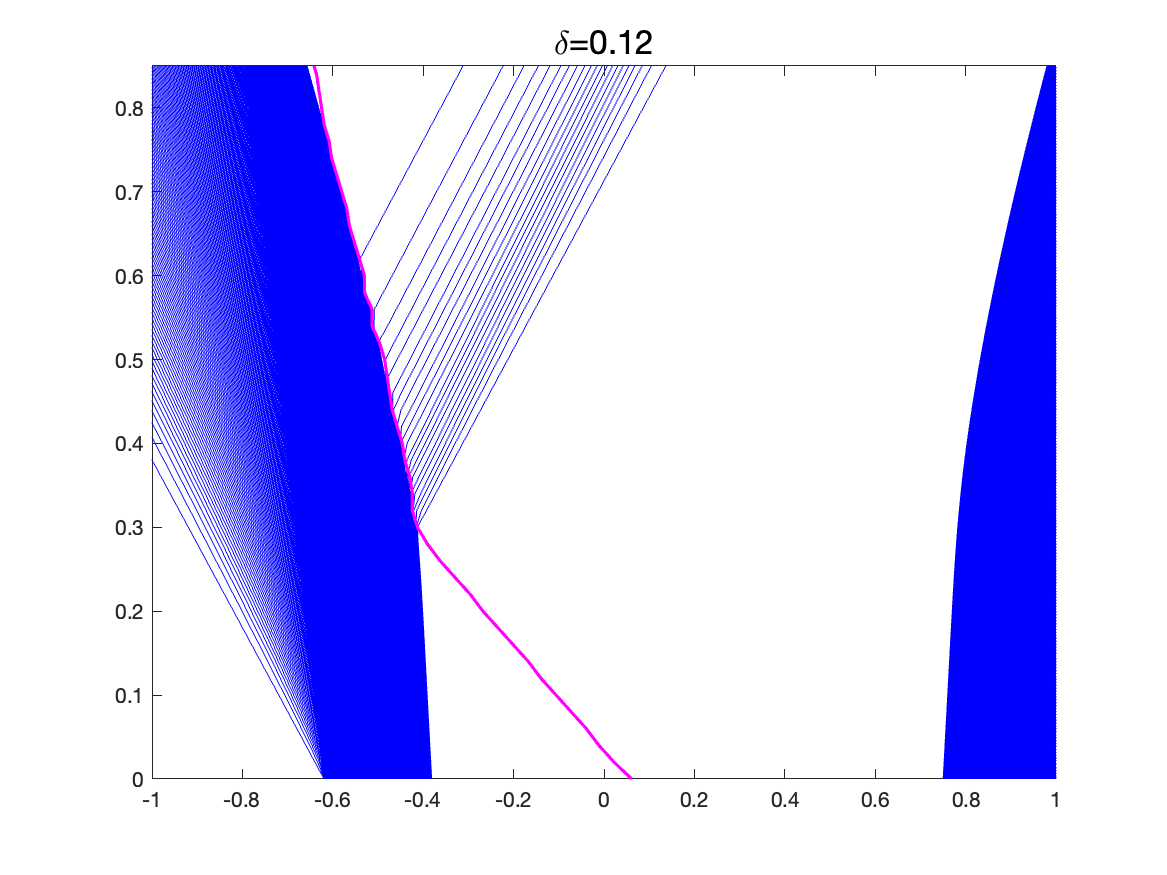}\\
\includegraphics[scale=0.25]{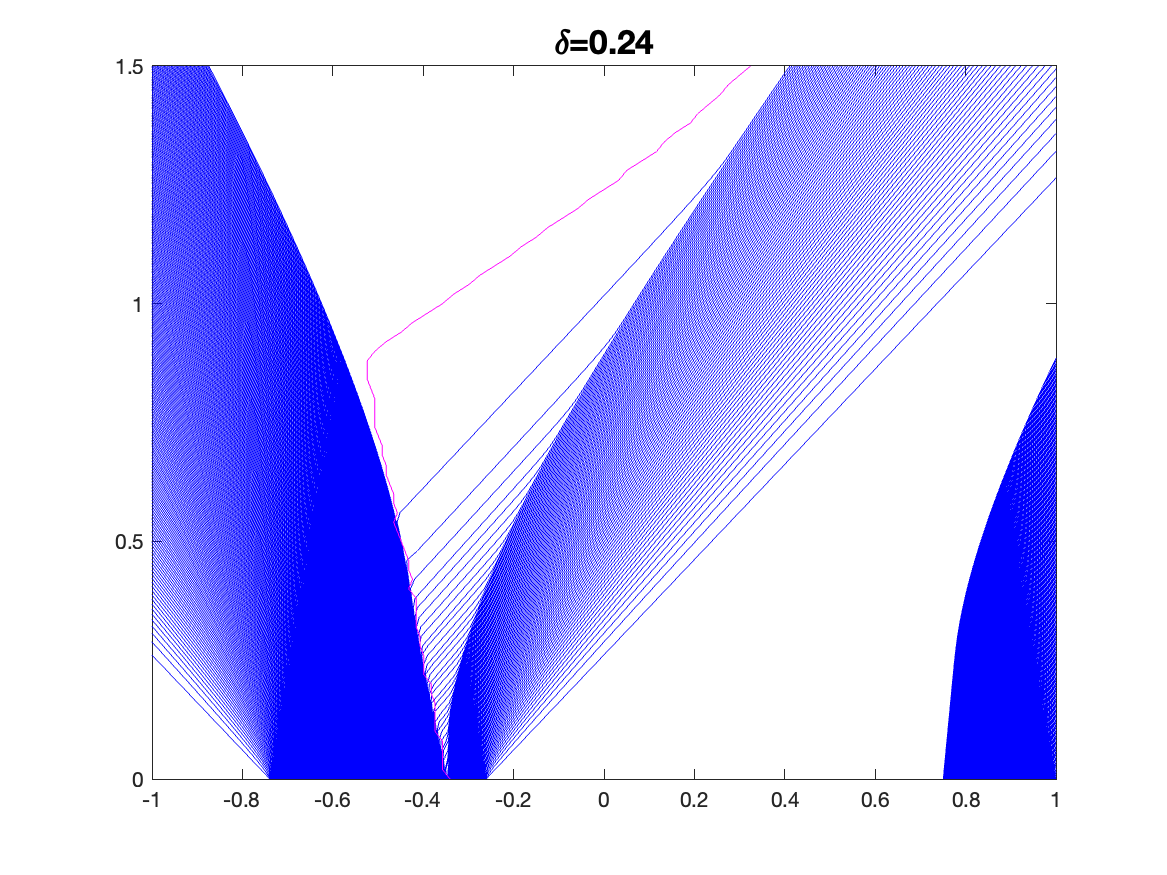}
\includegraphics[scale=0.25]{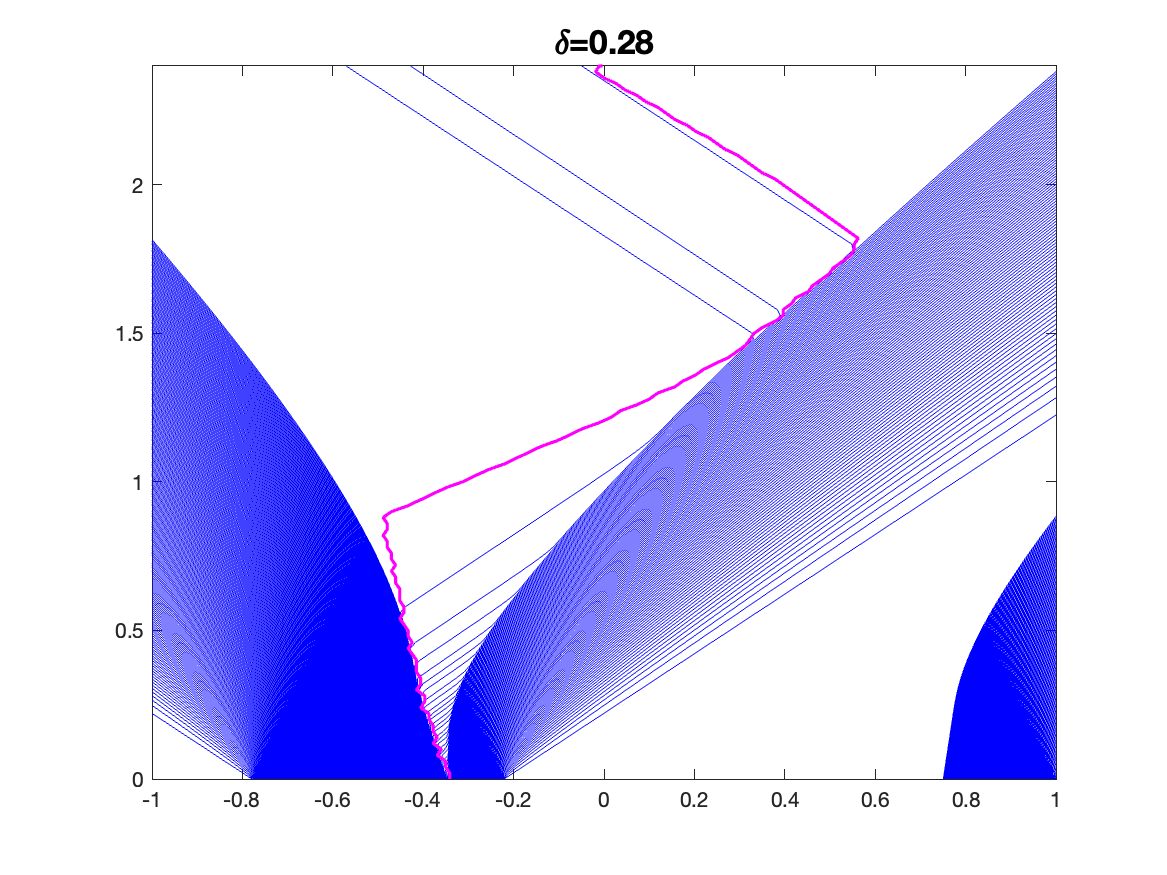}
\caption{Explanation of the discontinuities at $\delta = 0.1$ and $\delta=0.26$ of $T_{\rm mic}$ plotted in \figurename~\ref{fig:epsilon}. We select values of $\delta$ around $\delta=0.1$ (top) and $\delta=0.26$ (bottom). For $\delta=0.08$ no intersection between particles path (blue lines in the online version) and the turning curve (magenta line in the online version) occurs, whereas intersections are observed for $\delta=0.12$. A similar explanation applies to the case of $\delta=0.26$.}
\label{fig:crossing}
\end{figure}

\section*{Acknowledgments}
This paper has been supported by the RUDN University Strategic Academic Leadership Program.
SF and MDR are members of GNAMPA.
SF is partially supported by the Italian \lq\lq National Centre for HPC, Big Data and Quantum Computing\rq\rq, - Spoke 5 \lq\lq Environment and Natural Disasters\rq\rq, by the InterMaths Network, \url{www.intermaths.eu}, by the Ministry of University and Research (MIUR), Italy under the grant PRIN 2020 - Project N.~20204NT8W4, \lq\lq Nonlinear Evolutions PDEs, fluid dynamics and transport equations: theoretical foundations and applications\rq\rq, and by the INdAM project N.~E53C22001930001 \lq\lq MMEAN-FIELDSS\rq\rq.
MDR acknowledges the support by the INdAM-GNAMPA Project CUP E53C23001670001, and PRIN 2022 project \lq\lq Modeling, Control and Games through Partial Differential Equations\rq\rq\ (D53D23005620006), funded by the European Union - Next Generation EU.

\bibliographystyle{abbrv}
\bibliography{refe}

\end{document}